\documentclass{article}
\usepackage{amssymb}
\usepackage{mabliautoref}
\usepackage {amsmath}
\usepackage{amsthm}
\usepackage{graphicx}
\usepackage {amscd}
\usepackage {epic}
\usepackage[alphabetic]{amsrefs}
\usepackage{hyperref}

\newcommand{\sF}{\mathcal{F}}
\newcommand{\sO}{\mathcal{O}}

\newcommand{\bQ}{\mathbb{Q}}

\newcommand{\oZ}{\overline{Z}}

\DeclareMathOperator{\var}{var}

\DeclareMathOperator{\Supp}{Supp}

\begin{document}
\title{Ampleness of CM line bundle on the moduli space of canonically polarized varieties}

\author{Zsolt Patakfalvi and Chenyang Xu}
\date{\today}

\maketitle
\begin{abstract}We prove that the CM line bundle is ample on the proper moduli space which parametrizes KSBA stable varieties. 
\end{abstract}
\tableofcontents
\section{Introduction}
Through the note, we work over a ground field $k$  of characteristic zero. The moduli space  $M^{\rm can}$ of canonically polarized manifolds as well as  its natural geometric compactification have attracted considerable interest over the past few decades. It has been shown that geometric invariant theory (GIT) can be used to construct a moduli space of $M^{\rm can}$ (cf. \cite{Vie95, Don01}). However,  applying the GIT methods to construct a natural compactification is considerably more challenging  in higher dimension than in dimension 1 (see \cite{WX} for some new difficulties arising). On the other hand,  an alternative approach which uses the minimal model program (MMP) theory was first outlined in \cite{KSB} (see also \cite{Ale96}). Based on the recent progress in MMP, this strategy  turns out to give a satisfying compacitification ${M}^{\rm ksba}$. Although the coarse moduli space ${M}^{\rm ksba}$ first only exists as an algebraic space, \cite{Kollar90} has  developed a strategy to verify its 
projectivity, which was later completed in \cite{Fujino12} for the case of varieties and in \cite{KP} 
for pairs. 

We can apply the Knudson-Mumford's determinant construction to the sequence of ample line bundles on $M^{\rm ksba}$ constructed by Koll\'ar. The coefficient of the leading term is the CM line bundle (see Section \ref{ss-CM}), which was first introduced in \cite{Tian97} and later formulated in this way in \cite{PT10}. The curvature calculation on the Weil-Petersson metric of the CM line bundle suggests that it is ample. However, due to the presence of  possibly singular fibers, this is only completely worked out for $M^{\rm can}$ (see \cite{Sch12}). 

In this note we give a purely algebraic proof of the fact that the CM line bundle is ample on $M^{\rm ksba}$. 
Our approach is inspired by the recent interplay of studying $M^{\rm ksba}$ from both the algebraic and differential geometry view points, especially the equivalence of KSBA stability and K-stability for canonically polarized varieties (see \cite{Od10}). 

\begin{theorem}\label{thm-main}
The CM line bundle is ample on the KSBA moduli space. 
\end{theorem}
Using the Nakai-Moishezon criterion and the formula of the CM line bundle for a family of KSBA stable varieties, we immediately see that this is implied by the following theorem.
\begin{theorem}\label{t-big}
Let $T$ be a  normal variety and $(Z,\Delta)\to T$ a family of $n$-dimensional KSBA stable pairs with  finite fiber isomorphism equivalence classes $($see Definition \ref{def:varf}$)$, then $f_*((K_{Z/T}+\Delta)^{n+1}) $ is ample on $T$. 
\end{theorem}

At the end, we want to remark that the positivity of the CM line bundle is expected for spaces parametrizing K\"ahler-Einstein varieties or even polarized varieties with constant scalar curvature (see e.g. \cite{PT10}). But in general, not much is known.  In the case of the moduli space of K\"ahler-Einstein Fano varieties, it is still  not known how to show the positivity of the CM line bundle with only algebro-geometric tools. Using the fact that the curvature of the CM line bundle is the Weil-Petersson metric for a smooth family and some deep results in analysis, one can verify that it induces an embedding when restricting to the locus which parametrizes K\"ahler-Einstein Fano manfolds (see \cite{LWX15}). It remains to be an interesting and challenging question to 
prove similar results using only algebraic geometry. 

\bigskip

\noindent {\bf Acknowledgement: }  We thank Chi Li, Gang Tian and Xiaowei Wang for comments, discussions and references. We also want to thank the anonymous referee for useful suggestions on the exposition.

Partial financial support to CX was provided by The National Science Fund for Distinguished Young Scholars. A large part of this work was done while CX enjoyed the inspiring environment at the Institute for Advanced Studies, supported by Ky Fan and Yu-Fen Fan Membership Funds, S.S. Chern Fundation and NSF: DMS-1128155, 1252158.

ZP was supported by NSF grant DMS-1502236. 

\section{Preliminary}
\subsection{Notation and Conventions}
We follow the notation in \cite{KM98} and \cite{Kollar13}. 

When $X$ is a demi-normal (for the definition of demi-normal, see \cite[Definition 5.1]{Kollar13}) variety over $k$, we say that $(X,\Delta)$ is a pair if $\Delta=\sum a_i\Delta_i$ is an effective $\mathbb{Q}$-divisor with $a_i\le 1$, any component $\Delta_i$ is not contained in ${\rm Sing}(X)$ and $K_X+\Delta$ is $\mathbb{Q}$-Cartier. 

We refer to \cite[5.10]{Kollar13} for the definition of a pair $(X,\Delta)$ to be {\it semi log canonical} (slc).  Over an arbitrary base scheme $T$, let $f : X \to T$ be a flat family of slc models. We refer to \cite[28]{Kol13} for the definition of the $m$-th reflexive power $\omega^{[m]}_{X/T}$. 

\subsection{KSBA stable family}
In this section, we briefly introduce the concept of KSBA family. See \cite{Kollar15} for more background. 
\begin{definition}
A pair $(X, \Delta)$ over $k$  is {\it KSBA stable}, if it is proper over $k$, it  has slc singularities and $K_X + \Delta$ is ample. 
\end{definition}
We define the notion of a KSBA family in full generality only for the boundary free case, since that is what we need in \autoref{thm-main}. In the case of the presence of a boundary, we define the notion of KSBA family only over normal bases, as that is sufficient for the purposes of \autoref{t-big}. We want to note that when there is a boundary, the definition of a KSBA family over a general base is subtle (see \cite{Kollar15}).
\begin{definition}
For any scheme $T$ over $k$, a family $f : X \to  T$ is a KSBA stable family if $f$ is flat,  $X_t$ is KSBA stable for each $t \in T$, and it satisfies the following Koll\'ar condition: $\omega_{X/T}^{[m]}$ is compatible with base-change for each integer $m$, that is, if $S \to T$ is a morphism, then $\omega_{X_S/S}^{[m]} \cong \left( \omega_{X/T}^{[m]} \right)_S$.

A proper flat morphism $(X, D) \to T$ onto a normal variety is a KSBA stable family, if $D$ avoids the generic and the singular codimension one points of each fiber, $(X_t, D_t)$ is KSBA stable for each $t \in T$ and $K_{X/T} + D$ is $\bQ$-Cartier.
\end{definition}

\begin{remark}
The above two definitions are compatible. That is, if $X \to T$ is a KSBA stable family in the second sense then it is automatically a KSBA stable family in the first sense, i.e., satisfies the Koll\'ar condition, according to \cite[4.4]{Kollar11} and \cite[Cor 25]{Kollar08}.
\end{remark}

We also need the following definition from \cite[5.16]{KP}.
\begin{definition}[Variation]
\label{def:var} Given a  KSBA family $f \colon (X, D) \to T$  over an irreducible normal variety, such that the dimension $\dim(X_t) =n$ and the volume $(K_{X_t} +D_t)^n=v$. Let $I$ be the set of all possible sums, at most 1, formed from the coefficients of $D$. 
Then, there is an associated moduli map $\mu \colon Y \to \mathcal{M}_{n,v,I}$ to the moduli space of KSBA stable pairs with dimension $n$, volume $v$ and coefficient set $I$. The variation $\var(f)$ of $f$ is defined as the dimension of the image of $\mu$.

There is a more intuitive way to define $\var(f)$ that does not use the existence of the moduli space $\mathcal{M}_{n,v,I}$: let $f : (X, D) \to T$ be a KSBA stable family over a normal base, then we define the variation $\var(f)$ of the family to be $\dim T- d$, where $d$ is the dimension of a general isomorphism equivalence class of the fibers. 

Since the moduli map $\mu$ sends $t$ and $t'$ to the same point on $ \mathcal{M}_{n,v,I}$ if and only if
$(X_t, D_t)$ is isomorphic to $(X_{t'},D_{t'})$, the closed points of the
closed fibers of the moduli map are the equivalence classes under the
equivalence relation on the closed points of $T$ given by $t \equiv t'$ if
and only if  $(X_t, D_t)$ is isomorphic to $(X_{t'},D_{t'})$. Then the simple
addition formula for the dimension of the total space, the general
fiber and the base of a fibration yields that the dimension of the
image of the moduli map is $\dim T - d$, where $d$ is the dimension of a
general equivalence class as above. This says that the above two ways of defining $\var(f)$ are equivalent.  

If $\var f= \dim T$, that is a general fiber is isomorphic to only finitely many others, we say the family has maximal variation.
\end{definition}

\begin{definition}[Finite fiber isomorphism equivalence classes]
\label{def:varf} 
Let $f : (X, D) \to T$ be a KSBA stable family over a normal base of dimension $d$. We say that $f$ has finite fiber isomorphism equivalence classes, if for each $t \in T$, the set 
$$\{ u \in T | \ (X_u'',D_u'') \cong (X_t'',D_t'') \}$$ is finite.
\end{definition}

\begin{remark}
\label{rem:var_def}
By the existence of the Isom schemes of KSBA stable families \cite[Prop 5.8]{KP} there is an open set $U \subseteq T$, such that for every $u \in U$, the locus $\{t \in T| (X_t, D_t) \cong (X_u,D_u)\}$ is a locally closed subset of the same (general) dimension. 
\end{remark}

\begin{proposition} \cite[Cor 5.20]{KP} 
  \label{cor:extending_stable_log_families}
  Given $f : (X,D) \to T$ a family of stable log-varieties over a normal variety $T$, there is a generically finite proper map $T' \to T$ from a normal variety, another proper map $T' \to T''$ to a normal variety and a family of stable log varieties $f'' : (X'', D'') \to T''$ with maximal variation and finite fiber isomorphism equivalence classes such that the pullbacks of the above two families over $T'$ are isomorphic.
\end{proposition}

\begin{lemma}
\label{lem:restr_max_var} Given $f : (X,D) \to T$ a maximal variation KSBA stable family over a normal variety $T$ with $n$-dimensional fibers, and $H$ is an ample divisor on $X$, then 
$f_*(H^{n+1})$ is $\mathbb{Q}$-linearly equivalent to an effective cycle with the support $S$, such that $f_S:(X,D)\times_T S\to S$ is of maximal variation.  
\end{lemma}

\begin{proof}
Let us search for the above required  effective cycle $E$ in the form
 $$f_* \left(  H_1 \cap \cdots \cap H_{n+1}\right),$$ where $H_i \in |m H|$ for some integer $m \gg 0$. We are ready as soon as we make sure that each component of $E$ intersects $U$ of \autoref{rem:var_def}. Let $Z_l$ $(l=1,\dots t)$  be the components of $T \setminus U$. Then it is enough to guarantee that no component of $E$ is contained in  any of the $Z_l$. For that, just choose $H_i$ inductively to be general members of $|mH|$, which therefore does not contain any irreducible component of $f^{-1} Z_l \cap \left( \bigcap_{j=1}^{i-1} H_j \right) $. This way, $f^{-1} Z_l \cap \left( \bigcap_{i=1}^{n+1} H_i \right)$ will have dimension $\dim Z_l -1$, which shows that no component of $E$ is contained in $Z_l$.
\end{proof}

\subsection{CM line bundle}\label{ss-CM}
For the reader's convenience, in this section we recall some of the basic background on the CM line bundle. For more details see \cite{Tian97, PT10, FR06, PRS08, Wa12, WX}. The concept of the CM line bundle was first introduced in \cite{Tian97}. Unlike the Chow line bundle, it is not positive  on the entire Hilbert scheme (see \cite{FR06}). However, it is expected to be positive on the locus where the fibers are K-polystable.

Let $f : X \to B$ be a proper flat morphism of schemes of constant relative dimension $n \ge 1$ and let $A$ be a relatively ample line bundle on $X$. We will assume throughout that $B$ is normal, $X$ is $S_2$ and has pure dimension. We also assume that  $f$ has $S_2$, $G_1$ fibers.

Then Mumford-Knudsen's determinant bundle construction shows that there are line bundles $\lambda_0,..., \lambda_{n+1}$ such that the following formula holds:
$$\det f_{!} \left(A^{\otimes k} \right) =\det R^{\bullet}f \left(A^{\otimes k}\right) =\lambda^{ {k \choose n+1}}_{n+1}\otimes \lambda^{{k \choose n}}_{n} \otimes \cdots \lambda_0$$.

Let $\mu :=-\left(K_{X_t}\cdot A_{|_{X_t}}^{n-1}\right)/ A_{|_{X_t}}^{n}$, then 
$$\lambda_{\rm CM} = \lambda_{\rm CM} (X/B, A) := \lambda_{n+1}^{n\mu+n(n+1)}\otimes  \lambda_n^{-2(n+1)} .$$ 
A straightforward calculation using the Grothendieck-Riemann-Roch formula (see e.g. \cite{FR06}) shows that 
$$c_1(\lambda_{n+1})=f_*\left(c_1(A)^{n+1}\right)\qquad \mbox{and}\qquad nc_1(\lambda_{n+1}) -2c_1(\lambda_n)=f_*(c_1(A)^nc_1(K_{X/B})).$$
Hence
$$c_1(\lambda_{\rm CM})=f_*\left(n\mu c_1(A)^{n+1}+(n+1)c_1(K_{X/B})c_1(A)^n\right).$$
In particular, if $K_{X/B}$ is $\mathbb{Q}$-Cartier and relatively ample, let $A=K_{X/B}$, we simply have
$$c_1(\lambda_{\rm CM})=f_*\left((K_{X/B})^{n+1}\right).$$
Similarly, a log extension as in \cite{WX} shows that if we consider the log setting and assume $K_{X/B}+D$ to be  $\mathbb{Q}$-Cartier and relatively ample, let $A=K_{X/B}+D$, then
$$c_1\left(\lambda_{\rm CM}\left((X,D)/B\right)\right)=f_*\left( \left(K_{X/B}+D\right)^{n+1}\right)$$
(See \cite[2.8, 2.9]{WX}).

\subsection{Dlt blow up}\label{ss-dlt}

\begin{proposition}\label{p-dlt}
Let $g:(Z,\Delta)\to T$ be a KSBA stable family over a smooth variety $T$. We further assume that the generic fiber $(Z_t,\Delta_t)$ is log canonical. Then for each $0 < \epsilon \ll 1$, there is a pair $(X, D_\epsilon)$ and a divisor $0 \leq D$ on $X$ with a morphism $p : X \to Z$, such that 
\begin{enumerate}
\item \label{itm:equal} $K_X + D = p^* (K_Z + \Delta)$,
\item \label{itm:klt} $(X, D_\epsilon)$ is klt,
\item \label{itm:stable} $f : (X, D_\epsilon) \to T$ is a KSBA stable family, 
\item \label{itm:effective} $D - D_\epsilon$ is effective and its support is contained in ${\rm Ex}(p) \cap \Supp \left(p^{-1}_* \Delta^{=1} \right)$, and furthermore, 
\item \label{itm:variation} if the variation of $(Z, \Delta) \to T$ is maximal then so is the variation of $(X, D_\varepsilon)$.
\end{enumerate}

\end{proposition}
\begin{proof}
Let $\tilde{p}:\tilde{X}\to Z$ be a $\mathbb{Q}$-factorial dlt modification of $(Z,\Delta)$ \cite[1.36]{Kollar13} and write 
$$\tilde{p}^*(K_Z+\Delta)=K_{\tilde{X}}+\tilde{D}.$$ 
Denote by $\tilde{D}^{=1}=\lfloor \tilde{D} \rfloor$ and $\tilde{D}^{<1}=\tilde{D}-\tilde{D}^{=1}$. By \cite{BCHM10}, we may take $p : X \to Z$ to be  the relative log canonical model of $\left(\tilde{X},(1-\epsilon) \tilde{D}^{=1}+\tilde{D}^{<1}\right)$ over $Z$. Let $q: \tilde{X} \dashrightarrow X$  be the induced morphism and $D$, $D^{=1}$ and $D^{<1}$ the corresponding pushforwards. Define then $D_\epsilon:= (1-\epsilon) D^{=1} + D^{<1}$, whence \autoref{itm:klt} and \autoref{itm:effective} follows.  Note that by \cite[Thm E, p 414]{BCHM10}, $X$ is the same for all $0 < \epsilon \ll 1$.

Since $- \epsilon \tilde{D}^{=1} \equiv_Z K_{\tilde{X}} + (1- \epsilon)\tilde{D}^{=1}+\tilde{D}^{<1}$, we have that
$- \epsilon D^{=1} \equiv_Z K_{X} + D_\epsilon$ is ample over $Z$. Furthermore, 
\begin{multline}
\label{eq:compute_ampleness}
 K_{X}+D_\epsilon =q_* \left( K_{\tilde{X}} +(1-\epsilon) \tilde{D}^{=1}+\tilde{D}^{<1}\right)\\  = q_*  \left(q^* p^*(K_Z + \Delta) - \epsilon \tilde{D}^{=1}\right) = p^*(K_Z+\Delta)-\epsilon D^{=1}. 
\end{multline}
So, since  $K_Z+\Delta$ is ample over $T$,  $ K_{X}+D_\epsilon$ is ample over $T$ as well for $0 < \epsilon \ll 1$. A computation similar to \autoref{eq:compute_ampleness}, but with $D$ instead of $D_\epsilon$ shows \autoref{itm:equal}.

Let $n$ be the relative dimension of $Z$ over $Y$. To obtain (c), we need to show that $X$ is flat over $T$, $D_\epsilon$ does not contain any component of $X_t$ or any divisor in the singular locus of $X_t$ for any $t\in T$ and $(X_t, D_{\epsilon,t})$ is slc of dimension $n$ for all $t \in T$. We work on a neighborhood of an arbitrary point $t\in T$. 

To see the above statements, note first that if $W\subset X_t$ is a component, then $\dim W\ge  n$. Let $H_1,...,H_d$ be $d$ general hypersurfaces passing through $t$ where $d=\dim (T)$. Then as $(Z,\Delta+g^*(\sum^d_{i=1}H_i))$ is crepant birational to  $(X, D+f^*(\sum^d_{i=1} H_i))$, the latter is log canonical. As   $(X, D+f^*(\sum^d_{i=1} H_i))$ is log canonical at $W$, $f^* H_i$ are Cartier divisors and $W\subset \bigcap^d_{i=1}f^*H_i$, by \cite[34]{dFKX}, we know that $(X, f^*(\sum^d_{i=1} H_i))$ is snc at the generic point of $W$
and $W$ is not contained in $D_t$, which has the same support as $D_{\epsilon,t}$. Therefore $\dim W=n$ and $X_t$ is smooth at the generic point of $W$. In particular, $X_t$ is reduced and equidimensional.

 Since $(X,D_{\epsilon})$ is klt, then $X$ is Cohen-Maucaulay (see \cite[5.22]{KM98}). Since $f:X\to T$  is an equidimensional morphism and $T$ is smooth, $X_t$ is cut out by a regular sequence. In particular, $X_t$ is CM and $f$ is flat. Let $C=\bigcap^{d-1}_{i=1}H_i$, which is a smooth curve passing through $t$. Then $X_C\colon =X\times_T C$ is normal and 
 $$(X_C, D_{\epsilon,C}\colon= D_{\epsilon}\times_TC)$$
 satisfies that $(X_C, D_{\epsilon,C}+X_t)$ is log canonical by adjunction.
 This implies that $(X_C, D_{\epsilon,C})\to C$ is a KSBA family. Therefore, $D_\epsilon$ has to avoid the general point $\eta$ of any  codimensional one component  of the singular locus of $X_t$. Thus $(X_t, D_{\epsilon, t})$ is slc by  adjunction  and we conclude $(X,D_{\epsilon})\to T$ is a KSBA family over $T$.

To prove (e), just note that for a general $t \in T$, $(Z_t,\Delta_t)$ is the log canonical model of $(X_t, D_t)$. So, for general $t, u \in T$ and for $0 < \epsilon \ll 1$, $(X_t,D_{\epsilon,t}) \cong (X_u,D_{\epsilon,u})$ if and only if $(X_t,D_t) \cong (X_u,D_u)$, from which it follows that $(Z_t, \Delta_t) \cong (Z_u, \Delta_u)$. This shows \autoref{itm:variation}.

\end{proof}



\subsection{Bigness and nefness of relative canonical bundle}

We first collect some results about the push forwards of the powers of the relative canonical bundle. 

\begin{definition}
A torsion-free coherent sheaf $\sF$ on a normal variety $X$ is big, if for an ample line bundle $H$ on $X$, there is an integer $a>0$ and a generically surjective homomorphism $\bigoplus H \to S^{[a]}(\sF):=(S^a(\sF))^{**}$.
\end{definition}

\begin{theorem}
\label{thm:pushforward_big}
If $f : (X,D) \to T$ is a maximal variation KSBA stable family over a normal projective variety $T$ with klt general fibers, then $f_* \sO_X(r(K_{X/T}+D))$ is big for every sufficiently divisible integer $r>0$.
\end{theorem}
\begin{proof}This follows from \cite[Theorem 7.1]{KP}.
\end{proof}

\begin{remark}
Note that the klt assumption  in the above theorem cannot be weakened to log canonical according to \cite[Example 7.5-7.7]{KP}.
\end{remark}

\begin{theorem}
\label{thm:pushforward_nef}
If $f : (X,D) \to T$ is a KSBA stable family over a normal projective variety $T$, then $f_* \sO_X(r(K_{X/T}+D))$ is nef for every sufficiently divisible integer $r>0$.
\end{theorem}
\begin{proof}This follows from \cite[Theorem 1.13]{Fujino12}. \end{proof}

\begin{corollary}
\label{cor:CM_nef}
If $f : (X,D) \to T$ is a KSBA stable family over a normal projective variety $T$, then 
$$f_* \left((K_{X/T} + D)^{n+1} \right)$$ is nef.
\end{corollary}
\begin{proof}
Since $f_* \left((K_{X/T} + D)^{n+1} \right)$ is compatible with base-change, and nefness is decided on curves, we may assume that $T$ is a curve. However, then we are supposed to only prove that $0 \leq \deg f_* \left((K_{X/T} + D)^{n+1} \right)$, which follows if we show that $(K_{X/T} + D)^{n+1}$ is the limit of effective cycles.  The latter statement follows from the nefness of $K_{X/T}+D$.

Thus it suffices to show that $K_{X/T}+D$ is nef.  Since there is an embedding 
$$X\subset \mathbb{P}_T(f_* \sO_X(r(K_{X/T}+D)))$$
for $r$ sufficiently large and 
$$\mathcal{O}(1)|_X\cong r(K_{X/T}+D),$$
then the nefness of $K_{X/T}+D$ is a straightforward consequence of Theorem \ref{thm:pushforward_nef}.

\end{proof}

\begin{proposition}If $f : (X,D) \to T$ is a maximal variation KSBA stable family  over a smooth projective variety $T$ such that the generic fiber $(X_t,D_t)$ is log canonical, then $K_{X/T}+D$ is big and nef for every sufficiently divisible integer $r>0$.

\end{proposition}
\begin{proof} We have shown the nefness of $K_{X/T}+D$ in the last proof. 

\bigskip

For the bigness, when the general fiber is klt, by Theorem \ref{thm:pushforward_big}, there is a sequence of generically surjective morphisms
$$f^*(\oplus H)\to f^*S^af_* \sO_X(r(K_{X/T}+D))\to\mathcal{O}_X( ar(K_{X/T}+D)),$$
for some $a$ and sufficiently divisible $r$. 
After replacing $a$ by its multiple, and tensoring $A:=r(K_{X/T}+D)$, we know that there is a nontrivial morphism 
$$  f^*(\oplus mH)\otimes  \mathcal{O}_X(r(K_{X/T}+D)) \to\mathcal{O}_X( (ma+1)r(K_{X/T}+D)) ,$$
which implies that $K_{X/T}+D$ is big. 

In general, applying Proposition \ref{p-dlt}, we know that we can find a birational model $h:Y\to X$, such that if we define $D_Y$ so that 
$$h^*(K_X+D)=K_Y+D_Y$$ 
holds, then there exists a divisor $D'_Y\le D_Y$ for which $(Y,D'_Y)\to T$ is a maximal variation of KSBA stable family with generic klt fibers. Thus $K_{Y/T}+D'_Y$ is big which implies $K_{X/T}+D$ is big. 
\end{proof}

\section{Proof of Main Theorems}

In this section, we prove Theorem \ref{t-big} and hence Theorem \ref{thm-main}.  By induction,  we may assume the base is of dimension $d$ and Theorem \ref{t-big} already holds when the base is of dimension at most $d-1$. Note that the $d=0$ case is tautologically true, hence the starting point of the induction is fine. 

\subsection{Log canonical case}

\begin{lemma}\label{l-lc}
Theorem \ref{t-big} is true if a general fiber $(X_t,D_t)$ is log canonical. 
\end{lemma}
\begin{proof} Set  $A:=K_{X/T}+D$ and $d:=\dim (T)$. According to \autoref{cor:CM_nef}, by the Nakai-Moishezon criterion and resolution of singularities, it is enough to show that when $T$ is a $d$-dimensional smooth projective variety and $(X,D)/T$ is a KSBA family of maximal variation,  then  $\left(f_*\left(A^{n+1}\right)\right)^d>0$. 

Write $A=H+F$ for some ample $\mathbb{Q}$-divisor $H$ and effective $\mathbb{Q}$-divisor $F$.
According to \autoref{lem:restr_max_var},  $ f_* \left(H^{n+1}\right)$ is linearly equivalent to a $d-1$ dimensional effective $\mathbb{Q}$-cycle over whose support $(X,D)$ has a maximal variation. Hence by induction we have
$$\left(f_*\left(A^{n+1}\right)\right)^{d-1} \cdot f_*\left(H^{n+1}\right)>0. $$ 
So we only need to show that for any $0\le k\le n$,
$$\left(f_*\left(A^{n+1}\right)\right)^{d-1} \cdot f_*\left(H^{k}\cdot F \cdot A^{n-k}\right)\ge 0.$$

This again follows from the induction since  $f_*\left(H^{k}\cdot F \cdot A^{n-k}\right)$ is a limit of effective $(d-1)$-dimensional $\mathbb{Q}$-cycles. For any $d-1$ smooth projective variety $P\to T$,
by the projection formula
$$\left(g_*\left(A_{|_P}^{n+1}\right)\right)^{d-1} =\left(f_*\left(A^{n+1}\right)\right)^{d-1} \cdot P,$$
where $A_{|_P}$ denotes the restriction of $A$ on $X\times_T P$, and $g: X\times_T P\to P$ the natural morphism.  
Then by the induction we have
$$\left(f_*\left(A^{n+1}\right)\right)^{d-1} \cdot P\ge 0,$$
and this implies what we need by resolution of singularities.
\end{proof}

\subsection{Semi-log canonical case}
Let $f:(Z,\Delta)\to S$ be a KSBA stable family over a normal variety $T$. Taking a normalization $f:X\to Z$, we get
$$(X,D)=\sqcup^m_{i=1} (X_i,D_i)\to T$$
with a conductor divisor $E$ and $D_i$ is the sum of the conductor divisor and the pull back $\Delta_{X_i}$ of $\Delta$. Furthermore, there is an involution $\tau: E^{\rm n }\to E^{\rm n}$ on the normalization $E^{\rm n}$ of $E$ which preserves the difference divisor ${\rm Diff}_{E^{\rm n}}(\Delta_X)$. In fact, we know there is one to one correspondence between
$(Z,\Delta)/S$ and $(X,D, E, \tau)/S$ as above (see \cite[Theorem 5.13]{Kollar13}). 

\begin{lemma}
\label{l:normalization}
Let $f:(Z,\Delta)\to S$ be a KSBA stable family over a normal variety $T$, Taking a normalization $f:X\to Z$, we get
$$(X,D)=\sqcup^m_{i=1} (X_i,D_i)\to T.$$
Then $f_i:(X_i, D_i)\to T$ is a KSBA stable family over $T$.
\end{lemma}
\begin{proof}Let $0\in T$ be a closed point,  $(\bar{X}_{i,0},\bar{D}_{i,0})\to (X_{i,0},D_{i,0}) $ the normalization, and $\bar{D}_{i,0}$  the sum of the pull back of $D_{i,0}$ and the conductor divisor. Then $(\bar{X}_{i,0},\bar{D}_{i,0})$ is a KSBA stable pair.  

We first want to check $X\to T$ is flat at the codimension one point of $X_0$. In fact, to see this, by cutting the fiber using general hypersurfaces, we can assume that $Z\to T$ has relative dimension one, i.e., the fibers are nodal curves. 
Let $E$ be the conductor divisor of the normalization $g\colon X\to Z$. By definition, $g$ is isomorphic outside $E$. Furthermore, if a codimension one point $P$ of the fiber is contained in $E$, then $g(P)$ is a nodal point of the fiber, and analytically locally around $g(P)$, we can write it as $\hat{\mathcal{O}}_T[[x,y]]/ (xy-a)$ for some element $a\in \hat{\mathcal{O}}_T$. However, since $P\in E$, we conclude $a=0$,  which implies that $X\to T$ is smooth along $E$.

  
Thus we can apply the numerical stability in \cite[Section 13]{Kollar15} and conclude that for a general point $s\in T$, 
$$(K_{\bar{X}_{i,0}}+\bar{D}_{i,0})^n\ge (K_{\bar{X}_{i,s}}+\bar{D}_{i,s})^n .$$
Furthermore,  the equality holds if and only if $(X_i,D_i)\to T$ is a KSBA stable family over an open neighborhood of $0\in T$.

On the other hand, we know that
\begin{eqnarray*}
\sum^m_{i=1}(K_{\bar{X}_{i,s}}+\bar{D}_{i,s})^n  &= & (K_{Z_s}+\Delta_s)^n  \\
 & =&(K_{Z_0}+\Delta_0)^n\\
 & =&\sum^m_{i=1}(K_{\bar{X}_{i,0}}+\bar{D}_{i,0})^n,
\end{eqnarray*}
where the second equality follows from the fact that  $(Z,\Delta)$ is a stable family over $T$. Thus we conclude for each $i$,
$$(K_{\bar{X}_{i,0}}+\bar{D}_{i,0})^n= (K_{\bar{X}_{i,s}}+\bar{D}_{i,s})^n .$$ 
\end{proof}

\begin{remark}\label{r-normal}
We give a sketch of a more straightforward  argument for a weaker statement than Lemma \ref{l:normalization}, which says that there exists a proper dominant generically finite morphism $T'\to T$, such that the normalization 
$(X_i',D_i')$ of $(Z,\Delta)\times_T T'$ is a KSBA stable family over $T'$. This is enough for our calculation in the proof of Theorem \ref{t-big} for the general case. 

By generic flatness, there is an open set $T^0$, such that $(X^0_i,D^0_i):=(X_i,D_i)\times_T T^0\to T^0$ is a KSBA family over $T^0$. Applying \cite{AK00}, we can assume there is a proper dominant generically finite base change $g\colon T'\to T$ such that 
$$(X_i \times_T T', D_i \times_T T')$$ admits a weak semistable reduction. 
It follows from \cite[Theorem 1]{HX13} that by running a relative MMP of  $(X_i \times_T T', D_i \times_T T')$ over $T'$, we obtain a relative good minimal model $(X^{\rm m}_i,\Delta^{\rm m}_i)$. A similar argument as in the proof of Proposition of \ref{p-dlt} shows that $(X^{\rm m}_i,\Delta^{\rm m}_i)$ is flat over $T'$,  for any $t\in T'$, $\Delta^{\rm m}_{i,t}$ does not contain any component or codimension one singular point of $X^{\rm m}_{i,t}$, and $(X^{\rm m}_{i,t}, \Delta^{\rm m}_{i,t})$ is slc. 

Then the injectivity theorem (see \cite[Theorem 6.4]{Fujino13}) implies that the relative log canonical models $(X^{\rm c}_i,\Delta^{\rm c}_i)$ of $(X^{\rm m}_i,\Delta^{\rm m}_i)$ over $T'$ fiberwisely gives the log canonical model of $(X^{\rm m}_{i,t},\Delta^{\rm m}_{i,t})$. Furthermore, $(X^{\rm c}_i,\Delta^{\rm c}_i)$ is flat over $T'$ by Grauert's criterion. Thus $(X^c_i,\Delta^c_i) $ is a KSBA stable family over $T'$. 

Let $$(X',D')=\sqcup^m_{i=1} (X^{\rm c}_i,D^{\rm c}_i)\to T'.$$ Over $g^{-1}(T_0)$, $(X',D')$ extends the family of $(X,D)\times_{T}g^{-1}(T^0)$. 
Also we easily see both $E\times_{T} g^{-1}(T^0)$ and  the involution condition 
$$\tau\times_Tg^{-1}(T^0): (E\times_{T^0} g^{-1}(T^0))^{\rm norm}\to (E\times_{T^0} g^{-1}(T^0))^{\rm norm}$$ extend to corresponding data $E'$ and $\tau'$ over $T'$.

Thus 
$(X',D', E', \tau')/T'$ induces a KSBA stable family  $(Z',\Delta')/T'$ over $T'$ by \cite[Theorem 5.13]{Kollar13}. It satisfies that 
$$(Z',\Delta')\times_{T'} g^{-1}(T^0)= (Z,\Delta)\times_T g^{-1}(T^0).$$
By the separateness of the functor of KSBA stable family, we conclude that
$$(Z',\Delta')= (Z,\Delta)\times_T T',$$
as both  of them give KSBA stable families over $T$ which are isomorphic over the generic point. 
\end{remark}

\begin{proof}[Proof of \autoref{t-big}] 
We use the notations of \autoref{l:normalization}. Applying Proposition \ref{cor:extending_stable_log_families}, there is a smooth projective variety $T'$ with a generically finite morphism $T'$, with $m$ dominant morphisms $h_i : T'\to T_i$ to smooth projective varieties such that if we denote by $X_i' := X_i \times_T T'$  and  $ D_i':= D_i \times_T T'$, then
$$(X_i' ,  D_i') \cong (Y_i, E_i)\times_{T_i} T',$$
where $g_i : (Y_i,E_i)\to T_i$ are KSBA stable families over $T_i$ with finite fiber isomorphism equivalence classes and log canonical generic fibers. Furthermore, we have that
$$h : T'\to T_1\times\cdots \times T_m $$
is a generically finite morphism. 

According to \autoref{l-lc} and the induction on dimension, we know that $(g_i)_* \left( (K_{Y_i/T_i} + E_i)^{n+1} \right)$ is ample on $T_i$. 
Denote by $f_i' : X_i' \to T'$ and $p_i : T_1\times\cdots \times T_m  \to T_i$   the induced morphisms. Since $h$ is generically finite, and
\begin{multline*}
\left( f_* \left( (K_{X/T} + \Delta)^{n+1}  \right) \right)_{T'} = \sum_i (f_i')_* \left( ( K_{X_i'/T'} + D_i')^{n+1} \right) 
\\ = \sum_i h_i^* \left(g_{i*} \left( ( K_{Y_i/T_i} + E_i)^{n+1} \right) \right)= h^* \underbrace{\left( \sum_i p_i^* (g_i)_* \left( ( K_{Y_i/T_i} + E_i)^{n+1} \right)  \right)}_{\textrm{ample over $T_1\times\cdots \times T_m $}}
\end{multline*}
is big and nef on $T'$. Thus the above computation concludes the proof.
\end{proof}

\begin{proof}[Proof of Theorem \ref{thm-main}] We apply the Nakai-Moishezon criterion (cf. \cite{Kollar90}). So it suffices to check for any $d$-dimensional irreducible subspace $B\subset M^{\rm ksba}$, the top intersection $\lambda^d_{\rm CM}\cdot B>0$. 

By \cite[2.7]{Kollar90}, we can replace $B$ by a finite surjective base change $\pi: B'\to B$ such that $B'$ is normal and $B'\to M^{\rm ksba}$ lifts to $B'\to \mathcal{M}^{\rm ksba}$ where $\mathcal{M}^{\rm ksba}$ is the fine moduli  DM stack which parametrizes families of KSBA stable varieties. Thus there is a KSBA family of finite  fiber isomorphism equivalence classes $X/B'$. In particular,  
$$\lambda^d_{\rm CM}\cdot B=\frac{1}{\deg(\pi)}\lambda^d_{\rm CM}\cdot B'>0$$ by Theorem \ref{t-big}. 
\end{proof}

\begin{remark}A large part of our argument works in the log setting, i.e., for KSBA families of log pairs. However, due to the subtlety of the definition of the KSBA functor itself, we will not discuss it here. 
\end{remark}

\bigskip

{Department of Mathematics, Princeton University, Fine Hall, Washington Road, NJ 08544-1000, USA

E-mail address: pzs@math.princeton.edu}

\bigskip

{Beijing International Center of Mathematics Research, 5 Yiheyuan Road, Beijing 100871, China

E-mail address: cyxu@math.pku.edu.cn}

\end{document}